\documentclass{article}
\usepackage{amsmath, amsthm, amssymb}
\usepackage[dvips]{graphicx}
\usepackage{enumerate}
\def\authorsPS{}
\newcommand\authPS[8]{\ifnum\authPSc<1\def\authPSc{1}\else, \fi{\large\sffamily #1 #2}%
\edef\authorsPS{\authorsPS \par\vskip 1mm\noindent #1 #2:\hskip 5mm #3, #4, #5, #6, #7, #8}}

\newtheorem{theorem}{Theorem}[section]
\newtheorem{lemma}[theorem]{Lemma}

\newtheorem{proposition}[theorem]{Proposition}
\theoremstyle{definition}
\newtheorem{definition}[theorem]{Definition}
\theoremstyle{remark}

\def\authPSc{0}
\newcommand{\beq}[1]{
\begin{equation}\label{#1}}
\newcommand{\eeq}{\end{equation}}
\newcommand{\req}[1]{{\rm(\ref{#1})}}

\newcommand{\hten}{{\mathfrak H}}
\newcommand{\law}{\stackrel{\cal L}{\longrightarrow}}

\begin{document}


\title{CLT for an iterated integral with respect to fBm with H $>$ 1/2}

\author{Daniel Harnett\thanks{Corresponding author}, David Nualart\thanks{
D. Nualart is supported by the NSF grant DMS1208625. \newline
  Keywords:  Malliavin calculus, fractional Brownian motion.
  }   \\
  Department of Mathematics\thinspace ,\ University of Kansas\\
1460 Jayhawk Blvd\thinspace , \ Lawrence, Kansas 66045-7594 }
\maketitle

\begin{abstract}
We construct an iterated stochastic integral with fractional Brownian motion with H $>$ 1/2.  The first integrand is a deterministic function, and each successive integral is with respect to an independent fBm.  We show that this symmetric stochastic integral is equal to the Malliavin divergence integral.  By a version of the Fourth Moment theorem of Nualart and Peccati \cite{NP05}, we show that a family of such integrals converges in distribution to a scaled Brownian motion.  An application is an approximation to the windings for a planar fBm, previously studied by Baudoin and Nualart \cite{BN06}.  
   
\end{abstract}
\section{Introduction}
Let $B = \{(B_t^1, \dots, B_t^q), t\ge 0\}$ be a multidimensional fractional Brownian motion (fBm) with Hurst parameter $H>1/2$.  We study the asymptotic behavior as $k\to\infty$ of multiple stochastic integrals of the particular form:
\[ Y_{k^t} := \int_{[0,\infty)^q} s_q^{-qH} {\mathbf 1}_{\{1 \le s_1<\dots< s_q\le k^t\}} dB = \int_1^{k^t} \int_1^{s_q}\cdots\int_1^{s_2} s_q^{-qH} dB_{s_1}^1\dots dB_{s_{q-1}}^{q-1}~dB_{s_q}^q\]
where $t >0$ and each iterated integral is a pathwise symmetric integral in the sense of Russo and Vallois \cite{RV93}, and also a divergence integral.
Our main result is a central limit theorem for the process $\{ Y_{k^t}, t\ge 0\}$, namely that $\frac{Y_{k^t}}{\sqrt{\log k}}$ converges in distribution as $k\to \infty$ to a scaled Brownian motion.  Our approach uses the techniques of Malliavin calculus, where we express $Y_{k^t}$ in terms of the divergence integral $\delta$, which coincides with the multiple Wiener-It\^o stochastic integral in this case.  In our proof, convergence of finite-dimensional distributions follows from a multi-dimensional version of the Fourth Moment Theorem \cite{NP05, PT}, which gives conditions for weak convergence to a Gaussian random variable (see section 2.4).  Functional convergence to a Brownian motion is proved by investigating tightness.  In addition to the proof, we are able to comment on the rate of convergence (which is fairly slow: $\sim (\log k)^{-\frac 12}$), using a result from Nourdin and Peccati \cite{NoP11} in their recent book on the Stein method.

The original motivation for this paper was \cite{BN06}, where Baudoin and Nualart studied a complex-valued fBm with $H > 1/2$.  For $B = B_t^1 + iB_t^2$, $B_0=1$, they studied the integral
\beq{BN_1} \int_0^t \frac{dB_s}{B_s}\;=\;\int_0^t \frac{B_s^1dB_s^1+B_s^2dB_s^2}{|B_s|^2}~+~i\int_0^t\frac{B_s^2dB_s^1-B_s^1dB_s^2}{|B_s|^2}.\eeq
When $B$ is written in the form $\rho_te^{i\theta_t}$, the angle $\theta_t$ is given by the imaginary part of \req{BN_1}.  For standard Brownian motion, a well-known theorem by Spitzer \cite{Spitzer} holds that as $t\to \infty$, the random variable $2\theta_t/(\log t)$ converges in distribution to a Cauchy random variable with parameter 1.  We are not aware of a corresponding fBm version of Spitzer's theorem.  In \cite{BN06}, the functional
\beq{BN_2} Z_t := \int_1^t\frac{B_s^2dB_s^1-B_s^1dB_s^2}{s^{2H}}\eeq
was proposed as an asymptotic approximation for $\theta_t$.  It was shown (see Proposition 22 of \cite{BN06}) that $\frac{Z_t}{\sqrt{\log t}}$ converges in distribution to a Gaussian random variable, with an expression for variance similar to our own result.  Their proof also used Malliavin calculus, but did not use the Fourth Moment Theorem.  For $q=2$, since $B_t = \int_0^t dB_s$, $Z_t$ is asymptotically equal in law to
\[Z_t' = \int_1^t\int_1^s \frac{dB_r^2 dB_s^1}{s^{2H}} - \int_1^t\int_1^s \frac{dB_r^1dB_s^2}{s^{2H}},\]
 and we have a new (and shorter) proof of the result in \cite{BN06}.

The organization of this paper is as follows.  In Section 2, we give necessary details about the theoretical background; this includes a brief discussion of Malliavin calculus, and some remarks about integration with respect to fBm, which has been studied extensively elsewhere.  The Fourth Moment Theorem is also given.  In Section 3, we state and prove the main result, which is Theorem 3.2.  The proof follows in Sections 3.1 through 3.3; then Section 3.4 discusses the rate of convergence.  Section 4 contains a technical lemma which was put at the end due to length.    

\section{Notation and Theory}

We use the symbol ${\mathbf 1}_A$ for the indicator function of a set $A$.  Given a real $n$-tuple ${\mathbf x} := (x_1, \dots, x_n)$, we can re-arrange the variables in increasing order.  We denote this re-ordered $n$-tuple $(x_{(1)}, \dots, x_{(n)})$.  For a stochastic process $X = \{X(a), a\in {\cal I}\}$, we will use the notation $X_a$ and $X(a)$ interchangeably.  The symbol $C$ denotes a generic positive constant, which may change from line to line.

\subsection{Elements of Malliavin calculus}
Following is a brief description of some identities that will be used in the paper.  The reader may refer to  \cite{Nualart} for detailed coverage of this topic.  
Let $X = \{ X(h), h\in\cal{H}\}$ be an {\em isonormal Gaussian process} on a probability space $( \Omega, {\cal F}, P )$, and indexed by a Hilbert space $\cal{H}$.  That is, $X$ is a family of Gaussian random variables such that ${\mathbb E}[ X(h)] =0$ and ${\mathbb E}\left[X(h)X(g)\right] = \left< h,g\right>_{\cal{H}}$ for all $h,g \in\cal{H}$.
  
For integers $q \ge 1$, let ${\cal H}^{\otimes q}$ denote the $q^{th}$ tensor product of ${\cal H}$.  We use ${\cal H}^{\odot q}$ to denote the symmetric tensor product.  Given a function $f \in {\cal H}^{\otimes q}$, we define the symmetrization $\tilde f \in {\cal H}^{\odot q}$ as
\beq{symmetrize} \tilde{f}(x_1, \dots, x_q) = \frac{1}{q!}\sum_{\sigma} f(x_{\sigma(1)}, \dots , x_{\sigma(q)}),\eeq
where $\sigma$ includes all permutations of $\{ 1, \dots, q\}$.

Let $\{ e_n, n\ge 1\}$ be a complete orthormal system in ${\cal H}$.  For functions $f, g \in {\cal H}^{\odot q}$ and $p\in\{0, \dots, q\}$, we define the $p^{th}$-order contraction of $f$ and $g$ as that element of ${\cal H}^{\otimes 2(q-p)}$ given by
\beq{contract} f\otimes_p g = \sum_{i_1, \dots , i_p=1}^\infty \left< f, e_{i_1} \otimes \cdots\otimes e_{i_p}\right>_{{\cal H}^{\otimes p}} \otimes \left< g, e_{i_1} \otimes \cdots\otimes e_{i_p}\right>_{{\cal H}^{\otimes p}}\eeq
where $f\otimes_0 g = f\otimes g$ and $f\otimes_q g = \left< f,g\right>_{{\cal H}^{\otimes q}}$.  While $f$ and $g$ are both symmetric, the contraction may not be.  We denote its symmetrization by $f\stackrel{\sim}{\otimes}_p g$.

Let ${\cal H}_q$ be the $q^{th}$ Wiener chaos of $X$, that is, the closed linear subspace of $L^2(\Omega)$ generated by the random variables $\{ H_q(X(h)), h \in {\cal H}, \|h \|_{\cal H} = 1 \}$, where $H_q(x)$ is the $q^{th}$ Hermite polynomial, defined as
$$H_q (x) = {(-1)^q}e^{\frac{x^2}{2}}\frac{d^q}{dx^q}e^{-\frac{x^2}{2}}.$$  For $q \ge 1$, it is known that the map 
\beq{Hmap} I_q(h^{\otimes q}) = H_q(X(h))\eeq
provides a linear isometry between the symmetric product space ${\cal H}^{\odot q}$ (equipped with the modified norm $\sqrt{q!}\| \cdot\|_{{\cal H}^{\otimes q}}$) and ${\cal H}_q$, where $I_q(\cdot)$ is the Wiener-It\^o stochastic integral.  By convention, ${\cal H}_0 = \mathbb{R}$ and $I_0(x) = x$.  It follows from \req{Hmap} and the properties of the Hermite polynomials that for $f,g \in {\cal H}^{\odot q}$ we have
\beq{I_q_norm} {\mathbb E}\left[ I_q(f) I_q(g)\right] = q!\left< f,g\right>_{{\cal H}^{\otimes q}}.\eeq

Let $\cal S$ be the set of all smooth and cylindrical random variables of the form $F = g(X(\phi_1), \dots, X(\phi_n))$, where $n \ge 1$; $g: {\mathbb R}^n \to {\mathbb R}$ is an infinitely differentiable function with compact support, and $\phi_i \in {\cal H}$. The Malliavin derivative of $F$ with respect to $X$ is the element of $L^2(\Omega, {\cal H})$ defined as
$$DF = \sum_{i=1}^n \frac{\partial g}{\partial x_i}(X(\phi_1), \dots, X(\phi_n)) \phi_i.$$
By iteration, for any integer $q >1$ we can define the $q^{th}$ derivative $D^qF$, which is an element of $L^2(\Omega, {\cal H}^{\odot q})$.

We let ${\mathbb D}^{q,2}$ denote the closure of $\cal S$ with respect to the norm $\| \cdot \|_{{\mathbb D}^{q,2}}$ defined as 
$$\| F \|_{{\mathbb D}^{q,2}}^2 = {\mathbb E}\left[ F^2\right] + \sum_{i=1}^q {\mathbb E}\left[ \| D^iF \|_{{\cal H}^{\otimes i}}^2 \right].$$
We denote by $\delta$ the Skorohod integral, which is defined as the adjoint of the operator $D$.  A random element $u \in L^2(\Omega, {\cal H})$ belongs to the domain of $\delta$, Dom $\delta$, if and only if,
$$\left| {\mathbb E}\left[ \left< DF, u\right>_{\cal H}\right] \right| \le c_u \|F\|_{L^2(\Omega)}$$
for any $F \in {\mathbb D}^{1,2}$, where $c_u$ is a constant which depends only on $u$.  If $u \in $ Dom $\delta$, then the random variable $\delta (u) \in L^2(\Omega)$ is defined for all $F \in {\mathbb D}^{1,2}$ by the duality relationship,
$${\mathbb E}\left[ F\delta(u) \right] = {\mathbb E}\left[ \left< DF, u \right>_{\cal H} \right].$$
This is sometimes called the Malliavin integration by parts formula.  We iteratively define the multiple Skorohod integral for $q \ge 1$ as $\delta (\delta^{q-1}(u))$, with $\delta^0(u) = u$.  For this definition we have,
$${\mathbb E}\left[ F\delta^q(u) \right] = {\mathbb E}\left[ \left< D^qF, u \right>_{\hten^{\otimes q}} \right],$$
where $u \in$ Dom $\delta^q$ and $F \in {\mathbb D}^{q,2}$.  The adjoint operator $\delta^q$ is an integral in the sense that for a (non-random) $h \in {\cal H}^{\odot q}$, we have $\delta^q(h) = I_q(h)$. 

We will use the following hypercontractivity property of iterated integrals (see \cite{NP05}, Theorem 2.7.2, or \cite{Nualart}, Sec. 1.4.3 for complete details).  Let $f \in {\cal H}^{\odot q}$ and $p \ge 2$.  Then there exists a positive constant $C_{p,q} <\infty$, depending only on $p$ and $q$, such that
\beq{hyper} {\mathbb E}\left[ |I_q(f)|^p\right] \le C_{p,q}\left({\mathbb E}\left[ I_q(f)^2\right] \right)^{\frac p2}.\eeq
 
\subsection{Fractional Brownian motion}
Fix $T >0$ and an integer $d\ge 1$.  Let $B = \{B_t, 0\le t\le T\} = (B_t^1, \dots, B_t^d)$ be a $d$-dimensional fBm, that is, each $B_t^i$ is an independent, centered Gaussian process with $B_0^i = 0$ and covariance
\[{\mathbb E}\left[ B_s^i B_t^i\right] := R(s,t) = \frac 12 \left( s^{2H}+t^{2H}-|s-t|^{2H}\right)\]
for $t, s \ge 0$.  We assume that $\frac 12 < H < 1$.  
We will use the following elementary properties of $R(s,t)$:
\begin{enumerate}[(R.1)]
\item  $R(s,t) = R(t,s)$; and for any $\epsilon >0$,  $R(s+\epsilon,t) \ge R(s,t)$.
\item There are constants $1\le c_0 < c_1 \le  2$ such that $c_0 (st)^H \le R(s,t) \le c_1 (st)^H$.
\item As an alternate bound, if $s \le t$ then the Mean Value Theorem implies
\[R(s,t) \le s^{2H} + t^{2H} - (t-s)^{2H} \le s^{2H} + st^{2H-1}.\]
\end{enumerate}

Let $\cal E$ denote the set of ${\mathbb R}-$valued step functions on $[0,T]\times \{1,\dots,d\}$.  Note that any $f = f(t,i) \in \cal E$ may be written as a linear combination of elementary functions $e^k_t = {\mathbf 1}_{[0,t]\times\{k\}}$.  Let $\hten_d$ be the Hilbert space defined as the closure of $\cal E$ with respect to the inner product
\[ \left< e_s^k , e^j_t \right>_{\hten_d} = {\mathbb E}[ B_s^k B_t^j] = R(s,t)\delta_{kj},\]
where $\delta_{kj}$ is the Kronecker delta.  The mapping $e^k_t \mapsto B^k(t)$ can be extended to a linear isometry between $\hten_d$ and the Gaussian space spanned by $B$.  In this way, $\{ B(h), h \in \hten_d\}$ is an isonormal Gaussian process as in Section 2.1.

Let $\alpha_H = H(2H-1)$.  It is well known that
we can write
\beq{Rxy}R(s,t) = \alpha_H\int_0^s\int_0^t |\eta-\theta|^{2H-2}d\eta ~d\theta .\eeq
Consequently, for $f,g \in {\cal E}$ we can write
\beq{scalar_prod} \left< f, g\right>_{\hten_d} = \alpha_H \sum_{i=1}^d \int_0^T\int_0^T f(s,i)g(t,i)|t-s|^{2H-2}ds~dt.\eeq

We recall (see \cite{Nualart}, Sec. 5.1.3) that $\hten_d$ contains the linear subspace of measurable, ${\mathbb R}$-valued functions $\varphi$ on $[0,T]\times \{1,\dots, d\}$ such that
\[ \sum_{i=1}^d\int_0^T\int_0^T |\varphi(s,i)|~|\varphi(t,i)|~|t-s|^{2H-2}ds~dt < \infty.\]
We denote this space $|\hten_d|$.  Let $|\hten_d^{q,s}|$ be the space of symmetric functions $f: \left( [0,T]\times\{1,\dots, d\}\right)^q \to {\mathbb R}$ such that
\[ \sum_{i_1, \dots, i_q =1}^d \int_{[0,T]^{2q}} |f\left((\eta_1,i_1),\dots,(\eta_q,i_q)\right)|~|f\left((\theta_1,i_1),\dots,(\theta_q,i_q)\right)|~|\eta-\theta|^{2H-2}d\eta~d\theta~<\infty.\]
Then $|\hten_d^{q,s}|\subset \hten^{\odot q}$, and for $f,g \in |\hten_d^{q,s}|$ we can  write \req{contract} as
\beq{contract_int} f \otimes_p g = \alpha_H^p \sum_{k=1}^d \int_{[0,T]^{2p}} f\left(({\mathbf \eta},k),({\mathbf t_1},{\mathbf i_1})\right)g\left(({\mathbf \theta},k),({\mathbf t_2},{\mathbf i_2})\right) \prod_{j=1}^p |\eta_j - \theta_j|^{2H-2} d{\mathbf \eta}~d{\mathbf \theta},\eeq  
where \begin{align*}({\mathbf \eta},k) &= (\eta_1, k),\dots, (\eta_p,k);\; ({\mathbf \theta},k) = (\theta_1, k),\dots, (\theta_p,k);\;({\mathbf t_1},{\mathbf i_1}) = (t_1, i_1), \dots, (t_{q-p},i_{q-p}); \text{ and }\\({\mathbf t_2},{\mathbf i_2}) &= (t_{q-p+1}, i_{q-p+1}), \dots, (t_{2(q-p)},i_{2(q-p)}).\end{align*}

\subsection{Stochastic integration with respect to fBm}
Let  $F = g(B(\phi_1), \dots, B(\phi_n))$, where $n \ge 1$, $g: {\mathbb R}^n \to {\mathbb R}, \phi_i \in \hten_d$, and $g$ is a smooth function as in Section 2.1.  The Malliavin derivative of $F$ is an element of $\hten_d$ (which is isomorphic to the product space $(\hten_1)^d$), and we can write $D = (D^{(1)}, \dots,D^{(d)})$, where
\[D^{(i)}_t F = \sum_{j=1}^n \frac{\partial g}{\partial x_j}\left(B(\phi_1),\dots,B(\phi_n)\right)\phi_j(t,i),\]
where we use the notation $D_t^{(i)}F = D^{(i)}F(t)$.  We define the `component integral' $\delta^{(i)}$ as the adjoint of $D^{(i)}$, and use the notation
\beq{comp_int}\delta^{(i)}(u) = \int_0^T u_t \delta B^i_t;\;\;\text{ and}\eeq
\[\delta(u) = \int_0^T u_t \delta B_t  = \sum_{i=1}^d \delta^{(i)}(u).\]
where $u\in \text{Dom}~\delta^{(i)} \subset L^2(\Omega,\hten_1)$ for every $i=1,\dots, d$ implies $u\in \text{Dom}~\delta \subset L^2(\Omega,\hten_d)$. 

\medskip
The pathwise stochastic integral with respect to fBm with $H > 1/2$ has been studied extensively \cite{AN03,  Duncan, Nualart}.  For our purposes, we will use the symmetric Stratonovich integral discussed by Russo and Vallois \cite{RV93}:
\begin{definition}
For some $T > 0$, let $u = \{ u_t, 0\le t\le T\}$ be a stochastic process with integrable trajectories.  The symmetric integral with respect to the fBm $B$ is defined as 
\[\int_0^t u_s dB_s = \lim_{\varepsilon \downarrow 0} \frac{ 1}{2\varepsilon} \int_0^t u_s\left(B_{(s+\varepsilon)\wedge t} - B_{(s-\varepsilon)\vee 0}\right)ds,\]
where the limit exists in probability.\end{definition}

This theorem was first proved in \cite{AN03}.
\begin{theorem} Let $u = \{u_t, t\ge 0\}$ be a stochastic process in ${\mathbb D}^{1,2}(\hten_1)$ such that, for some $T > 0$,
\begin{align*}
&{\mathbb E} \left[ \int_0^T\int_0^T |u_t|~|u_s|~|t-s|^{2H-2}~ds~dt\right] < \infty;\\
&{\mathbb E}\left[ \int_{[0,T]^4} |D_tu_\theta |~|D_s u_\eta |~|t-s|^{2H-2}~|\theta-\eta |^{2H-2}du~dt~d\theta~d\eta\right] < \infty;\\
&\text{and }\;  \int_0^T\int_0^T |D_su_t|~|t-s|^{2H-2}~ds~dt < \infty\;\text{a.s.}\end{align*}
Then the limit of definition 2.1 exists in probability, and we have
\[ \int_0^T u_t dB_t = \int_0^T u_t \delta B_t + \alpha_H \int_0^T\int_0^T D_su_t |t-s|^{2H-2}ds~dt,\]
where $\alpha_H = H(2H-1)$.\end{theorem}

\subsection{The Fourth Moment Theorem}

\begin{theorem} Fix integers $n \ge 2$ and $d \ge 1$.  Let $\left\{ \left(f_1^{(k)},\dots ,f_d^{(k)}\right), k\ge 1\right\}$ be a sequence of vectors such that $f_i^{(k)} \in {\cal H}^{\odot n}$ for each $k$ and $i = 1, \dots, d$; and 
\begin{align*} \lim_{k \to \infty} n! \| f_i^{(k)}\|_{{\cal H}^{\otimes n}}^2 = \lim_{k \to \infty}\left\| I_n\left(f_i^{(k)}\right)\right\|_{L^2(\Omega)}^2 &= 
C_{ii}, \; \forall i = 1, \dots , d;\\
\lim_{k \to \infty} {\mathbb E} \left[ I_n\left(f_i^{(k)}\right) I_n\left(f_j^{(k)}\right)\right] &= C_{ij},\; \forall 1 \le i < j \le d.\end{align*}

Then the following are equivalent:
\begin{enumerate}[(i)]
\item As $k \to \infty$, the vector $\left( I_n(f_1^{(k)}),\dots, I_n(f_d^{(k)})\right)$ converges in distribution to a d-dimensional Gaussian vector with distribution ${\cal N}(0,{\mathbf C}_d)$;
\item  For each $i = 1, \dots, d$, $I_n(f_i^{(k)})$ converges in distribution to $N_i$, where $N_i$ is a centered Gaussian random variable with variance $C_{ii}$;
\item  For each $i = 1, \dots, d$,
\[ \lim_{k\to \infty} {\mathbb E}\left[ I_n\left( f_i^{(k)}\right)^4\right] = 3C_{ii}^2;\]
\item  For each $i = 1, \dots, d$, and each integer $1\le p \le n-1$,
 $\lim_{k \to \infty} \left\| f_i^{(k)} \otimes_p f_i^{(k)} \right\|_{{\cal H}^{\otimes 2(n-p)}} =0$.
\end{enumerate}\end{theorem}

This first version (which was 1-dimensional) of this theorem was proved in \cite{NP05}.  Since then, other equivalent conditions have been added  \cite{NoP11, NOrtiz}.  The multi-dimensional version stated above was proved by Peccati and Tudor \cite{PT}.  A key advantage of this theorem is that, unlike the standard method of moments, it is not necessary to know about moments of any order higher than four.
 
\medskip
\section{Main result}
Fix $q \ge 2$.  For $t > 0$ and integer $k \ge 2$, define
\[
Y_{k^t} = \int_1^{k^t} \int_1^{s_q}\cdots\int_1^{s_2} s_q^{-qH} dB_{s_1}^1\dots dB_{s_{q-1}}^{q-1}~dB_{s_q}^q, \]
where the stochastic integrals are iterated symmetric integrals in the sense of Definition 2.1.  Theorem 2.2 and the diagonal structure of $Y_{k^t}$ allow us to identify the pathwise and Skorohod integrals.

\begin{lemma} For each $q \ge 2$, we have
\beq{delta_Ykt}Y_{k^t}= \int_1^{k^t} \int_1^{s_q} \cdots \int_1^{s_2} s_q^{-qH}\delta B_{s_1}^1 \dots \delta B_{s_{q-1}}^{q-1}~\delta B_{s_q}^q.\eeq
\end{lemma}
\begin{proof}  This follows from iterated application of Theorem 2.2, where the correction term is zero due to independence.  Indeed, in the notation of \req{comp_int}, this is
\[Y_{k^t} = \delta^{(q)}\cdots \delta^{(1)}\left( s_q^{-qH} {\mathbf 1}_{\{1\le s_1 < \dots < s_q\le  k^t\}}\right).\]
\end{proof}

\medskip
Following is the main result of this section. 
\begin{theorem}For $t \ge 0$, define
\[X_k(0) =0; \quad
X_k (t) = \frac{Y_{k^t}}{\sqrt{\log k}}, \;t > 0.\]
Then as $k \to \infty$, the family $\{ X_k(t), t\ge 0\}$ converges in distribution to the process $X = \{ X(t), t\ge 0\}$, where $X$ is a scaled Brownian motion with variance $\sigma_q^2 t$, and
\beq{sigma_2}  \sigma_2^2 = \alpha_H \int_0^1 x^{-2H} R(1,x) (1-x)^{2H-2} dx; \;\text{ and for }q> 2,\eeq
\beq{sigma_q} \sigma_q^2 =\alpha_H^{q-1} \int_0^1 x_q^{-qH} (1-x_q)^{2H-2}\int_{\cal M} R(x_2, y_2)\prod_{i=2}^{q-1} |x_i-y_i|^{2H-2}~dx_2~dy_2 \dots dy_{q-1}~dx_q,\eeq
\[\text{where }\;{\cal M} = \{ 0\le x_2 < \dots < x_q; 0\le y_2 < \dots < y_{q-1}\le 1\}.\] 
\end{theorem}
The proof of Theorem 3.2 follows the lemmas in Sections 3.1 and 3.2.  Our first task is to investigate the covariance (Section 3.1), then verify two other conditions for weak convergence (Section 3.2).
 
\medskip 
\subsection{Convergence of the covariance function}
Let $A = \left\{ 1\le s_1 < \dots < s_q \le k^t\right\}$, and $B = \left\{( i_1, \dots, i_q) = (1,\dots, q)\right\}$.  Lemma 3.1 allows us to write $Y_{k^t} = \delta^q(f_{k^t})$, where
$f_{k^t} : \left( [0,\infty) \times \{ 1, \dots, q\}\right)^q \to {\mathbb R}$ 
is given by
\beq{f_k_t}
f_{k^t}\left( (s_1, i_1), \dots, (s_q, i_q)\right) = s_q^{-qH} {\mathbf 1}_A(s_1,\dots,s_q){\mathbf 1}_B(i_1,\dots,i_q).\eeq
Here, $f_{k^t} \in \hten^{\otimes q}$, where $\hten := \hten_q$ is the Hilbert space associated with a $q-$dimensional fBm (see Section 2.2).
Clearly, $f_{k^t}$ is not symmetric.  Instead, we will work with the symmetrization defined in \req{symmetrize}:
\beq{consum}\tilde{f}_{k^t} =\frac{1}{q!}\sum_\sigma s_{\sigma(q)}^{-qH} {\mathbf 1}_A(s_{\sigma(1)}, \dots,s_{\sigma(q)}){\mathbf 1}_B(i_{\sigma(1)},\dots,i_{\sigma(q)}), \eeq
where $\mathbf \sigma$ covers all permutations of $\{ 1, \dots, q\}$.  This gives equivalent results, by the relation $I_q(\tilde f) = I_q(f)$ (see \cite{Nualart}, Sec. 1.1.2).  

By definition $\tilde{f}_{k^t}$ is nonzero only if $1 \le s_{\sigma(1)} < \dots < s_{\sigma(q)} \le k^t$ and $(i_{\sigma(1)}, \dots, i_{\sigma(q)}) = (1,\dots,q)$, hence it is possible to express $\tilde{f}_{k^t}$ without a sum.  Let $\sigma$ be an arbitrary permutation of $\{ 1, \dots, q\}$, and let $A_\sigma = \{ 1 \le s_{\sigma(1)}, < \dots< s_{\sigma(q)} \le k^t\}$.  Since the sets $\{ A_\sigma\}$ form an almost-everywhere partition of $[1, k^t]^q$, we can write \req{consum} as 
\beq{sansum} \tilde{f}_{k^t} = \frac{1}{q!} s_{(q)}^{-qH}{\mathbf 1}_{A_1}\left( (s_1, i_1),\dots,(s_q,i_q)\right),\eeq
where $s_{(q)} = \max\{s_1, \dots, s_q\}$, and the set $A_1$ is defined by the following condition: when $s_1,\dots, s_q$ are arranged in $[1,k^t]$ such that $s_{(1)}< \dots < s_{(q)}$, then $(i_{(1)}, \dots, i_{(q)}) = (1, \dots, q)$.

\medskip
In the next three results, we check the conditions of Theorem 2.3 for $\delta^q( \tilde{f}_{k^t})$. 

\begin{lemma}  For each $q\ge 2$ and $t >0$,
\[t\sigma_q^2 = \lim_{k\to\infty} {\mathbb E}\left[ X_k(t)^2\right]\]
exists, where $\sigma_q^2$ is given by \req{sigma_2} and \req{sigma_q} for $q=2$ and $q>2$, respectively.\end{lemma}

\begin{proof}
Since $f_{k^t}$ is deterministic, we use \req{I_q_norm} and \req{scalar_prod}:
\[{\mathbb E}\left[ X_k(t)^2\right] = \frac{1}{\log k}{\mathbb E}\left[ \delta^q(f_{k^t})^2\right] = \frac{q!}{\log k}\left< \tilde{f}_{k^t}, \tilde{f}_{k^t} \right>_{\hten^{\otimes q}}\]
\beq{EX_k_1_a}=\frac{\alpha_H^q}{q!\log k} \int_{[1,k^t]^{2q}} (r_{(q)} s_{(q)})^{-qH}{\mathbf 1}_{A_1}({\mathbf r},{\mathbf i}){\mathbf 1}_{A_1}({\mathbf s},{\mathbf j}) \prod_{\ell=1}^q |r_{\ell}-s_{\ell}|^{2H-2}~d{\mathbf r}~d{\mathbf s},\eeq
where $({\mathbf r},{\mathbf i}) = \left((r_1,i_i), \dots ,(r_q, i_q)\right)$, and similar for $({\mathbf s},{\mathbf j})$.  
To evaluate \req{EX_k_1_a}, we decompose $[1,k^t]^{2q}$ into the union of the sets $\{ A_\sigma \times A_{\sigma'}\}$, which form a partition almost everywhere.  Since ${\mathbf 1}_{A_1}({\mathbf r},{\mathbf i})$ is nonzero only if $r_{\sigma(1)}<\dots<r_{\sigma(q)}$ and $(i_{\sigma(1)},\dots,i_{\sigma(q)}) = (1,\dots, q)$, and similar for ${\mathbf 1}_{A_1}({\mathbf s},{\mathbf j})$, it follows that we integrate only over the diagonal sets, that is, when $\sigma = \sigma'$.  Hence, \req{EX_k_1_a} can be integrated as a sum of $q!$ equal terms, and we have
\beq{E_X_k_1}
{\mathbb E}\left[ X_k(t)^2\right]
 = \frac{\alpha_H^q}{\log k} \int_{\cal A} (r_q s_q)^{-qH} \prod_{i=1}^q |r_i-s_i|^{2H-2}~dr_1~ds_1 \dots dr_q~ds_q,\eeq
where the integral is over the set
\[{\cal A} = \left\{1\le r_1< \dots< r_q\le k^t,~1\le s_1< \dots< s_q\le k^t\right\}.\]
Integrating over $r_1, s_1$, we have by L'H\^opital,
\begin{multline*}
\lim_{k\to\infty} \frac{\alpha_H^{q-1}}{ \log k} \int_{[1,k^t]^2}(r_q s_q)^{-qH} |r_q-s_q|^{2H-2} \int_{\cal A'} R(r_2, s_2)\prod_{i=2}^{q-1} |r_i-s_i|^{2H-2}~dr_2~ds_2 \dots dr_q~ds_q\\
= \lim_{k\to\infty} t{k^t\alpha_H^{q-1}} \int_1^{k^t} (r_qk^t)^{-qH} (k^t-r_q)^{2H-2}\int_{\cal A'} R(r_2, s_2)\prod_{i=2}^{q-1} |r_i-s_i|^{2H-2}~dr_2~ds_2 \dots ds_{q-1}~dr_q,
\end{multline*}
where the set ${\cal A'} = \{ 1\le r_2< \dots< r_q,~1\le s_2< \dots< s_{q-1}\le k^t\}$ ($\cal A'$ is empty if $q=2$).  Using the change of variable $r_i = k^tx_i, \; s_i = k^ty_i$, this may be written
\beq{lim_sigma} \lim_{k\to\infty} t\alpha_H^{q-1} \int_{\frac{1}{k^t}}^1 x_q^{-qH} (1-x_q)^{2H-2}\int_{\cal M} R(x_2, y_2)\prod_{i=2}^{q-1} |x_i-y_i|^{2H-2}~dx_2~dy_2 \dots dy_{q-1}~dx_q\eeq
\[ = t\alpha_H^{q-1} \int_0^1 x_q^{-qH} (1-x_q)^{2H-2}\int_{\cal M} R(x_2, y_2)\prod_{i=2}^{q-1} |x_i-y_i|^{2H-2}~dx_2~dy_2 \dots dy_{q-1}~dx_q,\]
where ${\cal M}$ is as in \req{sigma_q} for $q>2$, and we have \req{sigma_2} if $q=2$.  To show \req{sigma_2} and \req{sigma_q} are convergent, we use properties (R.1) and (R.2), so that
\[\sigma_2^2 = \alpha_H\int_0^1 x^{-2H}(1-x)^{2H-2}R(1,x)~dx \le c_1\alpha_H\int_0^1 x^{-H}(1-x)^{2H-2}~dx < \infty\]
and for $q > 2$
\[\sigma_q^2 \le \alpha_H\int_0^1 x_q^{-qH}(1-x_q)^{2H-2}R(1,x_q)^{q-1}dx_q \le c_1^{q-1}\alpha_H\int_0^1 x_q^{-H}(1-x_q)^{2H-2}~dx_q < \infty.\]
This concludes the proof.
\end{proof}

\begin{lemma}
Let $0 \le \tau \le t$.  For each $q \ge 2$,
\[\lim_{k\to\infty} {\mathbb E}\left[ X_k(t) X_k(\tau)\right] = \sigma_q^2 \tau;\]
and consequently $\lim_{k\to\infty}{\mathbb E}\left[ X(s)X(t)\right] = \sigma_q^2(s\wedge t)$ for all $0\le s,t < \infty$.
\end{lemma}

\begin{proof}
\begin{align*}
{\mathbb E}\left[X_k(t) X_k(\tau)\right] &={\mathbb E}\left[ (X_k(t) - X_k(\tau)+X_k(\tau))X_k(\tau)\right]\\
&=\frac{1}{\log k}{\mathbb E}\left[ (Y_{k^t} - Y_{k^\tau})Y_{k^\tau}\right] + {\mathbb E}\left[X_k(\tau)^2\right],
\end{align*}
where ${\mathbb E}\left[ X_k(\tau)^2\right] \to \sigma_q^2\tau$ by Lemma 3.3.  Note that
$Y_{k^t} - Y_{k^\tau} = \delta^q(\tilde{f}_{k^t}) - \delta^q(\tilde{f}_{k^\tau}),$ where, recalling the notation of \req{sansum},
\begin{multline*}\delta^q(\tilde{f}_{k^t}) - \delta^q(\tilde{f}_{k^\tau}) 
=\int_{[1,k^t]^{2q}} \frac{1}{q! s_{(q)}^{qH}} {\mathbf 1}_{A_1}({\mathbf s},{\mathbf i})~\delta B_s - \int_{[1,k^\tau]^{2q}} \frac{1}{q! s_{(q)}^{qH}} {\mathbf 1}_{A_1}({\mathbf s},{\mathbf i})~\delta B_s\\
=\int_1^{k^{t}}\int_1^{s_{(q)}}\cdots\int_1^{s_{(2)}} \frac{1}{q! s_{(q)}^{qH}} \delta B^{(1)}_{s_{(1)}}\cdots\delta B_{s_{(q-1)}}^{(q-1)}~\delta B_{s_{(q)}}^{(q)} -\int_1^{k^\tau}\int_1^{s_{(q)}}\cdots\int_1^{s_{(2)}} \frac{1}{q! s_{(q)}^{qH}} \delta B^{(1)}_{s_{(1)}}\cdots\delta B_{s_{(q-1)}}^{(q-1)}~\delta B_{s_{(q)}}^{(q)}\\=
\int_{k^\tau}^{k^{t}}\int_1^{s_{(q)}}\cdots\int_1^{s_{(2)}} \frac{1}{q! s_{(q)}^{qH}} \delta B^{(1)}_{s_{(1)}}\cdots\delta B_{s_{(q-1)}}^{(q-1)}~\delta B_{s_{(q)}}^{(q)}.
\end{multline*}
Hence, we can write $Y_{k^t} - Y_{k^\tau} = \delta^q(\tilde{f}_{\Delta k})$, where
\beq{f_Delta}
\tilde{f}_{\Delta k} = \frac{1}{q! s_{(q)}^{qH}}{\mathbf 1}_{A_1} {\mathbf 1}_{\{k^\tau \le s_{(q)} \le k^t\}}
= \tilde{f}_{k^t} {\mathbf 1}_{\{k^\tau \le s_{(q)} \le k^t\}}.\eeq
With this notation, it follows that
\begin{multline*}
\frac{1}{\log k} {\mathbb E}\left[(Y_{k^t} - Y_{k^\tau})Y_{k^\tau}\right] = \frac{q!}{\log k}\left< \tilde{f}_{\Delta k}, \tilde{f}_{k^\tau}\right>_{\hten^{\otimes q}} \\
= \frac{\alpha_H^q}{q!\log k}\int_{[1,k^t]^{2q}} \left(r_{(q)}s_{(q)}\right)^{-qH}{\mathbf 1}_{A_1}({\mathbf r},{\mathbf i}){\mathbf 1}_{A_1}({\mathbf s},{\mathbf j}){\mathbf 1}_{\{1\le s_{(q)} \le k^\tau \le r_{(q)} \le k^t\}}\prod_{\ell =1}^q |r_\ell -s_\ell |^{2H-2} d{\mathbf s}~d{\mathbf r}.\end{multline*}
As in Lemma 3.3, we decompose $[1,k^t]^{2q}$ into the union of the sets $\{ A_\sigma \times A_{\sigma'}\}$.  Since ${\mathbf 1}_{A_1}({\mathbf r},{\mathbf i})$ is nonzero only if $r_{\sigma(1)}<\dots<r_{\sigma(q)}$ and $(i_{\sigma(1)},\dots,i_{\sigma(q)}) = (1,\dots, q)$, and similar for ${\mathbf 1}_{A_1}({\mathbf s},{\mathbf j})$, it follows that we integrate only over the diagonal sets, that is, when $\sigma = \sigma'$.  Hence, we have $q!$ equal terms of the form
\beq{Cov_diff}
\frac{\alpha_H^q}{q!\log k} \int_{k^\tau}^{k^t}\int_1^{k^\tau}\int_{[1,k^t]^{2q-2}} \left(r_qs_q\right)^{-qH} {\mathbf 1}_{\{r_1<\dots< r_q\}}{\mathbf 1}_{\{s_1<\dots< s_q\}}\prod_{\ell =1}^q |r_\ell -s_\ell |^{2H-2} d{\mathbf s}~d{\mathbf r}.\eeq
By (R.1) and (R.2), for each $r_\ell \le r_q, ~ s_\ell \le s_q$, we have the estimate
\[
\alpha_H \int_1^{r_{(\ell)}} \int_1^{s_{(\ell)}} |r_{(\ell-1)} - s_{(\ell-1)}|^{2H-2} dr_{(\ell-1)}~ds_{(\ell-1)} \]
\beq{R_est}\qquad\qquad\le \alpha_H \int_0^{r_{(q)}}\int_0^{s_{(q)}} |r-s|^{2H-2}dr~ds = R(r_{(q)}, s_{(q)}) \le c_1 (r_qs_q)^H.\eeq
It follows that
\begin{align*}
\frac{1}{\log k} {\mathbb E}\left[(Y_{k^t} - Y_{k^\tau})Y_{k^\tau}\right]  &\le \frac{C}{\log k} \int_{k^\tau}^{k^t}\int_1^{k^\tau} (r_qs_q)^{-qH}R(r_q,s_q)^{q-1}|r_q-s_q|^{2H-2}dr_q~ds_q\\
&\le \frac{C}{\log k} \int_{k^\tau}^{k^t}\int_1^{k^\tau} (r_qs_q)^{-2H}R(r_q,s_q)|r_q-s_q|^{2H-2}dr_q~ds_q.
\end{align*}
Using the change-of-variable $s_q = k^\tau x,~r_q = k^\tau y$, this is bounded by
\[
\frac{C}{\log k} \int_1^{k^{t-\tau}}\int_0^1 (xy)^{-2H} R(x,y) (y-x)^{2H-2}dx~dy.\]
Using (R.3), we obtain the estimate,
\[\frac{C}{\log k} \int_1^{k^{t-\tau}}\int_0^1  \left(y^{-2H}(y-x)^{2H-2} + x^{1-2H}y^{-1}(y-x)^{2H-2}\right)~dx~dy,\]
where
\begin{align*}
\int_1^{k^{t-\tau}}\int_0^1 y^{-2H}(y-x)^{2H-2}dx~dy &\le \int_1^2 y^{-2H}\int_0^y (y-x)^{2H-2}dx~dy + \int_2^{k^{t-\tau}} (y-1)^{-2}dy\\
&\le C\int_1^2 y^{-1}dy + C\int_1^\infty y^{-2}dy < \infty,\end{align*}
and
\begin{align*}
\int_1^{k^{t-\tau}}\int_0^1 y^{-1}x^{1-2H}(y-x)^{2H-2}dx~dy &\le \int_1^2 y^{-1}\int_0^y x^{1-2H}(y-x)^{2H-2}dx~dy + \int_2^{k^{t-\tau}} (y-1)^{2H-3}dy\\
&\le C\int_1^2 y^{-1}dy + \int_1^\infty y^{2H-3}dy < \infty.\end{align*}
Hence, this term vanishes and Lemma 3.4 is proved.
\end{proof}
 
\medskip
\subsection{Conditions for weak convergence of $\{ X_k(t)\}$}
In the next two lemmas we verify additional properties of $\{ X_k(t)\}$.  In Lemma 3.5 we check condition (iv) of Theorem 2.3, and Lemma 3.6 is a tightness result.

\begin{lemma}  Fix $q \ge 2$ and $t>0$.  For each integer $1\le p \le q-1$,
\[ \lim_{k\to\infty} \;(\log k)^{-2} \| \tilde{f}_{k^t} \otimes_p \tilde{f}_{k^t} \|_{\hten^{\otimes 2(q-p)}}^2 = 0.\]\end{lemma}
\begin{proof}
Let $1\le p\le q-1$.  To compute the $p^{th}$ contraction of $\tilde{f}_{k^t}$, we use \req{contract_int}.
\beq{contr_1}
\tilde{f}_{k^t} \otimes_p \tilde{f}_{k^t} = \frac{\alpha_H^p}{(q!)^2} \int_{[1,k^t]^{2p}} (r_{(q)}s_{(q)})^{-qH} {\mathbf 1}_{A_1}({\mathbf r},{\mathbf i}){\mathbf 1}_{A_1}({\mathbf s},{\mathbf j}) \prod_{\ell =1}^p |r_\ell - s_\ell|^{2H-2}dr_1~ds_1 \dots dr_p~ds_p.\eeq
Using \req{contr_1}, we want to compute
\[\| \tilde{f}_{k^t} \otimes_p \tilde{f}_{k^t} \|_{\hten^{\otimes 2(q-p)}}^2 = \left< \tilde{f}_{k^t} \otimes_p \tilde{f}_{k^t}, \tilde{f}_{k^t} \otimes_p \tilde{f}_{k^t}\right>_{\hten^{\otimes 2(p-q)}}\]
\begin{multline}
= \frac{\alpha_H^{2q}}{(q!)^4}\int_{[1,k^t]^{4q}} \left(r_{(q)}s_{(q)}r_{(q)}'s_{(q)}'\right)^{-qH}
\left( {\mathbf 1}_{A_1}\right)^4 \prod_{i=1}^p \left(|r_i-s_i||r_i'-s_i'|\right)^{2H-2}\\
\times \prod_{i=p+1}^q \left(|r_i-r_i'||s_i-s_i'|\right)^{2H-2}~d{\mathbf r}~d{\mathbf s}~d{\mathbf r'}d{\mathbf s'}.\end{multline}
As in the proof of Lemma 3.3, we view integration over the set $[1,k^t]^{4q}$ as a sum of integrals over various cases corresponding to the orderings of the real variables $r_1, \dots, r_q$ (as in Lemma 3.3, the variables ${\mathbf s},{\mathbf r'},{\mathbf s'}$ must follow the same ordering).  Up to permutation of indices, each integral term has the form
\beq{contr_s}\frac{\alpha_H^{2q}}{(q!)^4}\int_{\cal G} (r_{(q)} s_{(q)} r_{(q)}' s_{(q)}')^{-qH} \prod_{i=1}^p \left(|r_i - s_i|~|r_i'-s_i'|\right)^{2H-2}
 \prod_{i=p+1}^q \left(|r_i - r_i'|~|s_i-s_i'|\right)^{2H-2} d{\mathbf r}~d{\mathbf s}~d{\mathbf r'}~d{\mathbf s'},\eeq
where ${\cal G} = \left\{ 1\le r_{(1)}< \dots < r_{(q)}\le k^t; \dots; 1\le s_{(1)}'< \dots < s_{(q)}'\le k^t\right\}$. 
To evaluate \req{contr_s}, there are two cases to consider.  The first case is if $r_{(q)} \in \{r_1, \dots, r_p\}$, that is, \req{contr_s} contains the terms $|r_{(q)}-s_{(q)}|, |r_{(q)}' - s_{(q)}'|$.  In this case, using \req{R_est} we can bound \req{contr_s} by
\begin{align*}\frac{\alpha_H^2}{(q!)^4}\int_{[1,{k^t}]^4}(r_{(q)} s_{(q)} r_{(q)}' s_{(q)}')^{-qH}& \left(R(r_{(q)},s_{(q)})R(r_{(q)}',s_{(q)}')\right)^{p-1} \left(R(r_{(q)},r_{(q)}')R(s_{(q)},s_{(q)}')\right)^{q-p}\\
&\qquad\;\; \times\left(|r_{(q)} - s_{(q)}|~|r_{(q)}'-s_{(q)}'|\right)^{2H-2}~dr_{(q)}~ds_{(q)}~dr_{(q)}'~ds_{(q)}'\end{align*}
\beq{contr_red}\le C\int_{[1,{k^t}]^4}(r s r' s')^{-2H} R(r,r')~R(s,s') \left(|r - s|~|r'-s'|\right)^{2H-2}dr~ds~dr'~ds',\eeq
where we used (R.2) in the last estimate.  The second case is the complement, that is, $r_{(q)} \in \{r_{p+1}, \dots, r_q\}$, so that \req{contr_s} contains the terms $|r_{(q)} - r_{(q)}'|,|s_{(q)}-s_{(q)}'|$.  If this is the case, then \req{contr_s} is bounded by
\begin{align*}\frac{\alpha_H^2}{(q!)^4}\int_{[1,{k^t}]^4}(r_{(q)} s_{(q)} r_{(q)}' s_{(q)}')^{-qH}& \left(R(r_{(q)},s_{(q)})R(r_{(q)}',s_{(q)}')\right)^{p} \left(R(r_{(q)},r_{(q)}')R(s_{(q)},s_{(q)}')\right)^{q-p-1}\\
&\qquad\;\; \times\left(|r_{(q)} - s_{(q)}|~|r_{(q)}'-s_{(q)}'|\right)^{2H-2}~dr_{(q)}~ds_{(q)}~dr_{(q)}'~ds_{(q)}'\end{align*}
\beq{contr_red_2}\le C\int_{[1,{k^t}]^4}(r s r' s')^{-2H} R(r,s)~R(r',s') \left(|r - r'|~|s-s'|\right)^{2H-2}dr~ds~dr'~ds'.\eeq
The result then follows by a change of variable and applying Lemma 4.1 to \req{contr_red} and \req{contr_red_2}.
\end{proof}

\medskip
\begin{lemma}There is a constant $0<C<\infty$ such that for each $k \ge 2$ and any $0 \le \tau < t < \infty$ we have
\[{\mathbb E}\left[ |X_k(t) - X_k(\tau)|^4\right] \le C(t-\tau)^2.\]\end{lemma}
\begin{proof}
Based on the hypercontractivity property \req{hyper}, it is enough to show
\[{\mathbb E}\left[ |X_k(t) - X_k(\tau)|^2\right] \le C(t-\tau).\]
Using the notation of \req{f_Delta}, we can write
\begin{multline*} {\mathbb E}\left[ |X_k(t) - X_k(\tau)|^2\right] = \frac{1}{\log k}{\mathbb E}\left[ |Y_{k^t} - Y_{k^\tau}|^2\right]= \frac{q!}{\log k}\left< \tilde{f}_{\Delta k}, \tilde{f}_{\Delta k}\right>_{\hten^{\otimes q}}\\
=\frac{\alpha_H^q}{q!\log k}\int_{[1,k^t]^{2q}} \left(r_{(q)}s_{(q)}\right)^{-qH}{\mathbf 1}_{A_1}({\mathbf r},{\mathbf i}){\mathbf 1}_{A_1}({\mathbf s},{\mathbf j}){\mathbf 1}_{\{k^\tau \le r_{(q)},s_{(q)} \le k^t\}}~\prod_{\ell =1}^q |r_\ell -s_\ell |^{2H-2} d{\mathbf s}~d{\mathbf r}.\end{multline*}
In the same manner as \req{Cov_diff}, this can be decomposed into a sum of $q!$ equal terms of the form
\[\frac{\alpha_H^q}{q!\log k} \int_{k^\tau}^{k^t}\int_{k^\tau}^{k^t}\int_{[1,k^t]^{2q-2}} \left(r_qs_q\right)^{-qH} {\mathbf 1}_{\{r_1\le\dots\le r_q\}}{\mathbf 1}_{\{s_1\le\dots\le s_q\}}\prod_{\ell =1}^q |r_\ell -s_\ell |^{2H-2} d{\mathbf s}~d{\mathbf r}.\]
Similar to Lemma 3.5, we use \req{R_est} and a change-of-variable to obtain
\begin{align*}
\frac{1}{\log k}{\mathbb E}\left[ |Y_{k^t} - Y_{k^\tau}|^2\right] &\le \frac{C}{\log k}\int_{k^\tau}^{k^t}\int_{k^\tau}^{k^t} (r_qs_q)^{-qH}R(r_q, s_q)^{q-1}|r_q-s_q|^{2H-2}dr_q~ds_q\\
&\le \frac{C}{\log k} \int_{k^{\tau-t}}^1\int_{k^{\tau-t}}^1 (xy)^{-2H} R(x,y) |x-y|^{2H-2}dx~dy.\end{align*}
Without loss of generality, assume $x < y$.  By (R.3), we have the estimate
\begin{align*}
\frac{C}{\log k}\int_{k^{\tau-t}}^1\int_{k^{\tau-t}}^1 (xy)^{-2H}R(x,y)& |x-y|^{2H-2}dx~dy = \frac{C}{\log k}\int_{k^{\tau-t}}^1\int_{k^{\tau-t}}^y (xy)^{-2H}R(x,y) |x-y|^{2H-2}dx~dy\\
&\le \frac{C}{\log k}\int_{k^{\tau-t}}^1\int_0^y y^{-2H}(y-x)^{2H-2} + x^{1-2H}y^{-1}(y-x)^{2H-2} dx~dy\\
&\le \frac{C}{\log k} \int_{k^{\tau-t}}^1 y^{-2H}y^{2H-1} + y^{-1}~dy \le C(t-\tau).\end{align*}
This concludes the proof.
\end{proof}

\medskip
\subsection{Proof of Theorem 3.2}
Fix integers $q \ge 2$ and $d\ge 1$, and choose a set of times $0\le t_1 < \dots < t_d$.  Lemmas 3.3 and 3.4 show that the random vector sequence $\left\{\left(X_k(t_1),\dots,X_k(t_d)\right), k\ge 1\right\}$ meets the covariance conditions of Theorem 2.3.  Moreover, Lemma 3.5 verifies condition (iv) of Theorem 2.3.  Therefore, we conclude that as $k \to \infty$,
\beq{Vec_in_law}  \left(X_k(t_1),\dots,X_k(t_d)\right) \law \left(X(t_1),\dots, X(t_d)\right),\eeq
where each $X(t_i)$ has distribution ${\cal N}(0, \sigma_q^2 t_i)$, and ${\mathbb E}\left[ X(t_i) X(t_k)\right] = \sigma_q^2(t_i \wedge t_k)$ for all $1 \le i,k \le d$.  By Lemma 3.6, the sequence $\{X_k(t)\}$ is tight, hence it follows from \req{Vec_in_law} that the sequence converges in the sense of finite-dimensional distributions, and we conclude that the family $\{X_k(t), t\ge 0\}$ converges in distribution to the process $\{ X(t), t\ge 0\} \stackrel{\cal L}{=} \{\sigma_q W_t, t\ge 0\},$ where $W_t$ is a standard Brownian motion.  This concludes the proof of Theorem 3.2.

\medskip
\subsection{Rate of convergence}
Let $t >0$ be fixed.  By Theorem 3.2, it follows that the sequence $\{X_k(t), k\ge 1\}$ converges in distribution to a random variable $N(t)$, where $N(t) \sim {\cal N}(0,\sigma_q^2 t)$.  Recent work by Nourdin and Peccati \cite{NoP11} has produced a stronger form of the Fourth Moment Theorem for the 1-dimensional case, that is, that the conditions of the Fourth Moment Theorem also imply convergence in the sense of total variation (as well as other metrics - see Theorem 5.2.6).  The result below follows from Corollary 5.2.10 of \cite{NoP11}.

\begin{proposition}  Let $t \ge 0$.  Then for sufficiently large $k$, there is a constant $0 < C < \infty$ such that
\[ d_{TV} \left( X_k(t), N(t)\right) \le \frac{C}{\sqrt{\log k}},\]
where $d_{TV}(\cdot,\cdot)$ is total variation distance.  Hence $X_k(t)$ converges as $k\to\infty$ to Gaussian in the sense of total variation.
\end{proposition}
\begin{proof}The result follows from an estimate in \cite{NoP11} (Cor. 5.2.10):

\beq{stein_dtv} d_{TV} \left(X_k(t),N(t)\right) \le 2 \sqrt{\frac{{\mathbb E}\left[ X_k(t)^4\right]-3{\mathbb E}\left[ X_k(t)^2\right]^2}{3{\mathbb E}\left[ X_k(t)^2\right]^2}} + \frac{2 \left| {\mathbb E}\left[ X_k(t)^2\right]-\sigma^2_q t\right|}{{\mathbb E}\left[ X_k(t)^2\right]\vee \sigma_q^2 t}.\eeq
To simplify notation, we will assume $t=1$.  To help interpret this estimate, the following identity is computed in \cite{NoP11} (see Lemma 5.2.4):
\beq{stein_dist} {\mathbb E}\left[ X_k(1)^4\right]-3{\mathbb E}\left[ X_k(1)^2\right]^2
= \frac{3}{q(\log k)^2} \sum_{p=1}^{q-1} p(p!)^2 \binom{q}{p}^4 (2q-2p)! \| \tilde{f}_k \stackrel{\sim}{\otimes}_p \tilde{f}_k\|^2_{\hten^{\otimes 2(q-p)}}.\eeq
From Lemma 3.5, we know $(\log k)^{-2}\| \tilde{f}_k \stackrel{\sim}{\otimes}_p \tilde{f}_k\|^2_{\hten^{\otimes 2(q-p)}} \to 0$ at a rate $C/\log k$, hence it follows the first term of \req{stein_dtv} is of order $C(\log k)^{-\frac 12}$.  The second term depends on the covergence rate of \req{E_X_k_1}.  In the proof of Lemma 3.3, convergence follows from a limit of the form ${{\mathbb E}\left[ Y_k^2\right]}/{\log k}$.  By L'H\^opital's rule, it follows the rate of convergence has the form $C/\log k$, hence the first term controls. 
\end{proof}

\medskip
\section{A technical lemma}

\begin{lemma}Fix $T> 0$.  Let $1/2 < H < 1$, and for nonnegative $x, y$, let $R(x,y) = \frac 12 \left( x^{2H} + y^{2H}-|x-y|^{2H}\right)$.  Then there is a constant $0< K<\infty$ such that
\[ \int_{[\frac 1T, 1]^4} (xyuv)^{-2H} R(x,y)~R(u,v) |x-u|^{2H-2}|y-v|^{2H-2}dx~dy~du~dv \le K \log T.\]\end{lemma}

\begin{proof}
In the following computations, we will obtain estimates based on the order of integration. Due to the symmetries of the integral, it is enough to consider four distinct cases.  We will make frequent use of (R.3), and for a second estimate, note that for $x < y < u$ we can write $(u-x)^{2H-2} \le (u-y)^{-\alpha}(y-x)^{-\beta}$, where $\alpha, \beta > 0$ satisfy $\alpha + \beta = 2-2H$.

\medskip
{\em Case 1:}  $x \le y \le u \le v$
We can write
\begin{align*}
\int_{\frac 1T}^1 \int_{\frac 1T}^v &(uv)^{-2H} R(u,v) \int_{\frac 1T}^u y^{-2H}(v-y)^{2H-2}\int_{\frac 1T}^y x^{-2H}R(x,y)(u-x)^{2H-2}~dx~dy~du~dv\\
&\le C\int_{\frac 1T}^1 \int_{\frac 1T}^v (uv)^{-2H} R(u,v) \int_{\frac 1T}^u y^{-2H}(v-y)^{2H-2}(u-y)^{-\alpha}\int_{\frac 1T}^y (y-x)^{-\beta} + x^{1-2H}y^{2H-1}(y-x)^{-\beta}~dx\dots dv\\
&\le C\int_{\frac 1T}^1 \int_{\frac 1T}^v (uv)^{-2H} R(u,v) (v-u)^{-\alpha}\int_{\frac 1T}^u y^{1-2H-\beta}(u-y)^{-\beta-\alpha}~dy~du~dv\\
&\le C\int_{\frac 1T}^1 \int_{\frac 1T}^v (uv)^{-2H} R(u,v) (v-u)^{-\alpha}u^{-\beta}~du~dv\\
&\le C\int_{\frac 1T}^1 v^{-2H} \int_{\frac 1T}^v u^{-2H}(v-u)^{-\alpha}\left( u^{2H} +uv^{2H-1}\right) ~du~dv\\
&\le C\int_{\frac 1T}^1 v^{-1} dv \le K \log T.\end{align*}

\medskip
{\em Case 2:} $x<y<v<u$
For this case, we use constants $\alpha, \beta > 0$ such that $\alpha+\beta = 2H-2$, and $\gamma,\delta > 0$ such that $\gamma + \delta = \alpha$.
\begin{align*}
\int_{\frac 1T}^1 \int_{\frac 1T}^u &(uv)^{-2H} R(u,v) \int_{\frac 1T}^u y^{-2H}(v-y)^{2H-2}\int_{\frac 1T}^y x^{-2H}R(x,y)(u-x)^{2H-2}~dx~dy~dv~du\\
&\le C\int_{\frac 1T}^1 \int_{\frac 1T}^u (uv)^{-2H} R(u,v) \int_{\frac 1T}^u y^{1-2H-\beta}(v-y)^{2H-2}(u-y)^{-\alpha}~dy~dv~du\\
&\le C\int_{\frac 1T}^1 \int_{\frac 1T}^u (uv)^{-2H} R(u,v)(u-v)^{-\gamma} \int_{\frac 1T}^u y^{1-2H-\beta}(v-y)^{2H-2-\delta}~dy~dv~du\\
&\le C\int_{\frac 1T}^1 u^{-2H} \int_{\frac 1T}^u v^{-2H-\beta-\delta}\left( v^{2H} + vu^{2H-1}\right)(u-v)^{-\gamma}~dv~du\\
&\le C\int_{\frac 1T}^1 u^{-1} \le K\log T.\end{align*} 

\medskip
{\em Case 3:}  $x<u<y<v$
\begin{align*}
\int_{\frac 1T}^1 \int_{\frac 1T}^v &(yv)^{-2H} (v-y)^{2H-2} \int_{\frac 1T}^y u^{-2H}R(u,v)\int_{\frac 1T}^u x^{-2H}R(x,y)(u-x)^{2H-2}~dx~du~dy~dv\\
&\le C\int_{\frac 1T}^1 \int_{\frac 1T}^v (yv)^{-2H} (v-y)^{2H-2} \int_{\frac 1T}^y u^{-2H}\left(u^{2H} +uv^{2H-1}\right)\left(u^{2H-1} +y^{2H-1}\right)~du~dy~dv\\
&\le C\int_{\frac 1T}^1 \int_{\frac 1T}^v (yv)^{-2H} (v-y)^{2H-2} \int_{\frac 1T}^y \left(u^{2H-1}+y^{2H-1}+v^{2H-1}+u^{1-2H}(vy)^{2H-1}\right)~du~dy~dv\\
&\le C\int_{\frac 1T}^1 \int_{\frac 1T}^v (yv)^{-2H} (v-y)^{2H-2}\left( y^{2H} + yv^{2H-1}\right)~dy~dv\\
&\le C\int_{\frac 1T}^1 v^{-1}~dv \le K \log T.\end{align*}

\medskip
{\em Case 4:}  $x < v < u<y$
\begin{align*}
\int_{\frac 1T}^1 \int_{\frac 1T}^y &(uy)^{-2H}\int_{\frac 1T}^u v^{-2H}R(u,v)(y-v)^{2H-2}\int_{\frac 1T}^v x^{-2H}R(x,y)(u-x)^{2H-2}~dx~dv~du~dy\\
&\le C\int_{\frac 1T}^1 \int_{\frac 1T}^y (uy)^{-2H}\int_{\frac 1T}^u v^{-2H}R(u,v)(y-v)^{2H-2}(u-v)^{-\alpha}\int_{\frac 1T}^v x^{-2H}\left(x^{2H} + xy^{2H-1}\right)(v-x)^{-\beta}~dx~dv~du~dy\\
&\le C\int_{\frac 1T}^1 \int_{\frac 1T}^y (uy)^{-2H}(y-u)^{-\alpha}\int_{\frac 1T}^u v^{-2H}\left(v^{2H} + vu^{2H-1}\right)(u-v)^{-\alpha-\beta}\left(v^{1-\beta}+v^{2-2H-\beta}y^{2H-1}\right)~dv~du~dy\\
&\le C\int_{\frac 1T}^1 \int_{\frac 1T}^y (uy)^{-2H}(y-u)^{-\alpha}\left( u^{2H-\beta} +y^{2H-1}u^{1-\beta}\right)~du~dy\\
&\le C\int_{\frac 1T}^1 y^{-2H}\left( y^{1-\alpha-\beta} + y^{2H-1}\right)~dt\\
&\le C\int_{\frac 1T}^1 y^{-1}~dy \le K \log T.\end{align*}

\end{proof}

\end{document}